\documentclass[12pt]{amsart}

\usepackage{amssymb, amsfonts}
\usepackage{url}
\usepackage[all,arc]{xy}
\usepackage{amscd}
\usepackage{mathrsfs}
\usepackage{geometry}
\usepackage[colorlinks,citecolor=blue]{hyperref}
\usepackage{comment}
\usepackage{enumerate}
\usepackage{graphicx}
\usepackage{bm}
\usepackage{color}
\usepackage{anyfontsize}
\usepackage{caption}
\usepackage{dirtytalk}

\numberwithin{equation}{section}

\theoremstyle{plain}

\newtheorem{theorem}{Theorem}[section]
\newtheorem{corollary}[theorem]{Corollary}
\newtheorem{lemma}[theorem]{Lemma}
\newtheorem{proposition}[theorem]{Proposition}

\newtheorem{definition}[theorem]{Definition}
\newtheorem{remark}[theorem]{Remark}

\begin{document}


\title[Static manifolds with boundary]{Static manifolds with boundary: Their geometry and some uniqueness theorems}

\author{Vladimir Medvedev}


\address{Faculty of Mathematics, National Research University Higher School of Economics, 6 Usacheva Street, Moscow, 119048, Russian Federation}

\email{vomedvedev@hse.ru}



\begin{abstract}
Static manifolds with boundary were recently introduced to mathematics. This kind of manifold appears naturally in the prescribed scalar curvature problem on manifolds with boundary when the mean curvature of the boundary is also prescribed. They are also interesting from the point of view of general relativity. For example, the (time-slice of the) photon sphere on the Riemannian Schwarzschild manifold splits it into static manifolds with boundary. In this paper, we prove a number of theorems that relate the topology and geometry of a given static manifold with boundary to some properties of the zero-level set of its potential (such as connectedness and closedness). Also, we characterize the round ball in the Euclidean 3-space with standard potential as the only scalar-flat static manifold with mean-convex boundary whose zero-level set of the potential has Morse index one. This result follows from a general isoperimetric inequality for 3-dimensional static manifolds with boundary, whose zero-level set of the potential has Morse index one. Finally, we prove some uniqueness theorems for the domains bounded by the photon sphere on the Riemannian Schwarzschild manifold.
 \end{abstract}

\maketitle


\newcommand\cont{\operatorname{cont}}
\newcommand\diff{\operatorname{diff}}

\newcommand{\dvol}{\text{dA}}
\newcommand{\Ric}{\operatorname{Ric}}
\newcommand{\Hess}{\operatorname{Hess}}
\newcommand{\GL}{\operatorname{GL}}
\newcommand{\myO}{\operatorname{O}}
\newcommand{\myP}{\operatorname{P}}
\newcommand{\eye}{\operatorname{Id}}
\newcommand{\myF}{\operatorname{F}}
\newcommand{\Vol}{\operatorname{Vol}}
\newcommand{\odd}{\operatorname{odd}}
\newcommand{\even}{\operatorname{even}}
\newcommand{\ol}{\overline}
\newcommand{\mye}{\operatorname{E}}
\newcommand{\myo}{\operatorname{o}}
\newcommand{\myt}{\operatorname{t}}
\newcommand{\irr}{\operatorname{Irr}}
\newcommand{\mydiv}{\operatorname{div}}
\newcommand{\curl}{\operatorname{curl}}
\newcommand{\re}{\operatorname{Re}}
\newcommand{\im}{\operatorname{Im}}
\newcommand{\can}{\operatorname{can}}
\newcommand{\scal}{\operatorname{scal}}
\newcommand{\tr}{\operatorname{trace}}
\newcommand{\sgn}{\operatorname{sgn}}
\newcommand{\SL}{\operatorname{SL}}
\newcommand{\myspan}{\operatorname{span}}
\newcommand{\mydet}{\operatorname{det}}
\newcommand{\SO}{\operatorname{SO}}
\newcommand{\SU}{\operatorname{SU}}
\newcommand{\specl}{\operatorname{spec_{\mathcal{L}}}}
\newcommand{\fix}{\operatorname{Fix}}
\newcommand{\id}{\operatorname{id}}
\newcommand{\grad}{\operatorname{grad}}
\newcommand{\singsup}{\operatorname{singsupp}}
\newcommand{\wave}{\operatorname{wave}}
\newcommand{\ind}{\operatorname{ind}}
\newcommand{\mynull}{\operatorname{null}}
\newcommand{\inj}{\operatorname{inj}}
\newcommand{\arcsinh}{\operatorname{arcsinh}}
\newcommand{\Spec}{\operatorname{Spec}}
\newcommand{\Ind}{\operatorname{Ind}}
\newcommand{\Nul}{\operatorname{Nul}}
\newcommand{\inrad}{\operatorname{inrad}}
\newcommand{\mult}{\operatorname{mult}}
\newcommand{\Length}{\operatorname{Length}}
\newcommand{\Area}{\operatorname{Area}}
\newcommand{\Ker}{\operatorname{Ker}}
\newcommand{\floor}[1]{\left \lfloor #1  \right \rfloor}

\newcommand\restr[2]{{
  \left.\kern-\nulldelimiterspace 
  #1 
  \vphantom{\big|} 
  \right|_{#2} 
  }}


\section{Introduction}

The main problem in General Relativity is to solve the Einstein equations. However, this problem is transcendentally complicated. One of the simplest cases where it becomes possible to solve Einstein's equations is the so-called \textit{static space-times}. Recall that a Lorentzian manifold is called a static space-time if it can be locally represented as the Lorentzian warped product of a complete Riemannian manifold $(M,g)$ with the real line, i.e., there exists a warping function $V\in C^\infty(M)$ such that the Lorentzian metric takes the form $\mathfrak{g}=-V^2dt^2+g$ on $M\times\mathbb R$. The function $V$ is called a \textit{static potential or lapse function}. We will use the term "static potential" and omit the word "static". A straight-forward computation shows that a static space-time satisfies the \textit{Einstein vacuum equation with cosmological constant} if, and only if, the warping function on the underlying Riemannian manifold satisfies the following equation
\begin{equation}\label{eq:static}
\Hess_{g}V-(\Delta_{g} V){g}-V\Ric_{g}=0,
\end{equation}
where $\Hess_g$ denotes the Hessian, $\Delta_g=\tr_g\Hess_g$ is the Laplace-Beltrami operator of the metric $g$, and $\Ric_g$ is its Ricci curvature. 

\begin{definition}
A Riemannian manifold $(M^n,g)$, admitting a non-trivial smooth solution $V\colon M\to\mathbb{R}$ to~\eqref{eq:static}, is called a \textit{static manifold}. 
\end{definition}

General relativity is not the only place where static manifolds appear. Another example, when equation~\eqref{eq:static} appears in geometry, comes from the observation that the formal $L^2$-adjoint operator to the linearized scalar curvature operator takes the form $DR^*|_g(V)=\Hess_{g}V-(\Delta_{g} V){g}-V\Ric_{g}$. Such an operator appears naturally in the \textit{prescribed scalar curvature problem} (see, for example, \cite{fischer1975deformations,bourguignon1975stratification,corvino2000scalar}). Cruz and Vitorio in~\cite{cruz2019prescribing} considered a similar problem on \textit{manifolds with boundary}. This question led them to considering the formal $L^2$-adjoint operator to the linearized  operator, which assigns to a metric $g$ on a given manifold with boundary $M$ the pair $(R_g,H_g)$, where $R_g$ is the scalar curvature on $(M,g)$ and $H_g$ is the mean curvature of $\partial M$. A function $V$, that belongs to the kernel of this operator has to satisfy the following boundary value problem
\begin{equation}\label{static}
\left\{
   \begin{array}{rcl}
\Hess_{g}V-(\Delta_{g} V){g}-V\Ric_{g} &= &0\quad\mbox{in}\quad M,\\
\dfrac{\partial V}{\partial \nu}g-VB_{g}& = &0\quad\mbox{on}\quad\partial M.
\end{array}
   \right.
\end{equation}
Here $\nu$ denotes the outward unit normal vector field to $\partial M$ and $B_g$ is the second fundamental form of $\partial M$ with respect to $\nu$. This motivates the following definition (see~\cite{almaraz2022rigidity})

\begin{definition} A \emph{static manifold with boundary} is a triple $(M,g,V)$, where $(M,g)$ is a Riemannian manifold with boundary and $V$ is a non-trivial smooth solution to~\eqref{static}.  
 \end{definition}
 
 If $(M,g,V)$ is a static manifold with boundary, then we still call $V$ a (static) potential. 
 
 \begin{remark}
 For further developments in the prescribed scalar curvature problem on manifolds with boundary, when the mean curvature is also prescribed, see the following papers~\cite{ho2020deformation,cruz2023critical,sheng2024localized,sheng2024static}.
 \end{remark}
 
 The reader can easily verify that system~\eqref{static} is equivalent to the following equations
 \begin{equation}\label{static1}
\Hess_gV=V\left(\Ric_g - \dfrac{R_g}{n-1}g\right)\ \ \mbox{and} \ \ \Delta_gV=-\dfrac{R_g}{n-1}V \ \ \mbox{on $M$}
\end{equation}
and 
\begin{equation}\label{static2}
V\left(B_g-\dfrac{H_g}{n-1}g\right)=0 \ \ \mbox{and} \ \ \dfrac{\partial V}{\partial \nu}=\dfrac{H_g}{n-1}V \ \ \mbox{on $\partial M$}.
\end{equation}

Static manifolds with boundary have many interesting properties. For example, the scalar curvature $R_g$ of such a manifold is a constant (see~\cite[Proposition 1]{cruz2023static}). In the standard way, we may assume that $R_g=\epsilon n(n-1)$, where $\epsilon\in\{1,0,-1\}$. Sometimes, the scalar curvature is written in the form $R_g=(n-1)\Lambda$ and $\Lambda$ is called the \textit{cosmological constant}. We will also use this notation.

It is interesting to classify static manifolds with boundary. In the paper~\cite{cruz2023static} the authors obtained some uniqueness theorems for compact scalar-flat static manifolds with mean-convex boundary under the assumption that the zero-level set of the potential $\Sigma$ is connected. Our first result tries to answer the question under which geometrical/topological assumptions can we guarantee that the zero-level set of the potential is connected ( $Int~M$ denotes the interior of $M$). 

\begin{theorem}\label{cor:closed1}
If a compact static manifold with boundary $(M^n,g,V)$ with non-positive cosmological constant is a topological cylinder and $V^{-1}(0)=\Sigma\subset Int~M$, then $\Sigma$ is connected.  
\end{theorem}

This is the case of the compact domains on the Schwarzschild and Schwarzschild-anti de Sitter manifolds described in \textit{Examples 6} and \textit{7} in Section~\ref{sec:examples}. Notice that in the case where the cosmological constant is positive, Theorem~\ref{cor:closed1} fails in general (see \textit{Example 10} in Section~\ref{sec:examples}). 

However, it turns out that it is not always possible to satisfy the condition $\Sigma\subset Int~M$. That is, the following theorem holds.

\begin{theorem}\label{cor:closed2}
If a compact static manifold with boundary $(M^n,g,V)$ with non-positive cosmological constant has one boundary component, then $\Sigma=V^{-1}(0)$ intersects $\partial M$ or $V$ does not vanish in $M$. 
\end{theorem} 

This is the case of geodesic balls in Euclidean and hyperbolic spaces. The potentials are defined in \textit{Examples 2} and \textit{4} in Section~\ref{sec:examples}. 

\begin{remark}
An analog of Theorem~\ref{cor:closed1} for the case where $\Sigma$ intersects the boundary is also proved in the present paper (see Theorem~\ref{thm:bd} below).
\end{remark}

Our next result is the inverse to Theorem~\ref{cor:closed1} in the following sense: Here we assume that the zero-level set of the potential is connected, and we find geometrical assumptions under which the static manifold has two boundary components.

\begin{theorem}\label{thm:rigidity}
Let $(M^n,g,V)$ be a compact static manifold with boundary with $R_g=\epsilon n(n-1)$, which satisfies one of the following conditions

\begin{enumerate}[(i)]

\item $\epsilon=1$ and the mean curvature of any boundary component is non-negative;

\item $\epsilon=0$ and the mean curvature of any boundary component is positive.

\end{enumerate}

Suppose that $\Sigma=V^{-1}(0)\subset Int~M$ is connected. Then, respectively,  

\begin{enumerate}[(i)]

\item $\partial M$ has at most two boundary components;

\item $\partial M$ has exactly two boundary components.

\end{enumerate}

Moreover, if $V=const$ on $\partial M$, then any boundary component has positive scalar curvature. In particular, when $n=3$ the boundary components are round spheres. 
\end{theorem}

This theorem is a straightforward generalization of Theorem 3 in~\cite{cruz2023static}. Note that not all possible combinations of $R_g$ and $H_g$ are feasible for a compact static manifold with boundary (see~\cite[Table 1]{sheng2024static}).

Next, using the famous Bunting-Masood-ul-Alam argument (see~\cite{bunting1987nonexistence}) in the same way as in the proof of~\cite[Theorem 3]{cruz2023static}, we obtain the following corollary.

\begin{corollary}
Let $(M^3,g,V)$ be a compact static manifold with boundary, satisfying the case $(ii)$ in the previous theorem, $V=const$ on $\partial M$ and $\Sigma=V^{-1}(0)\subset Int~M$ is connected. Then $(M^3,g)$ is the doubling of $([2m,3m]\times \mathbb S^2,g_m)$, with $m=2/(3\sqrt{3}H_S)$, and 
$$
V=V_m(r)=\dfrac{1-\dfrac{m}{2r}}{1+\dfrac{m}{2r}},
$$ 
the Schwarzschild potential in the isotropic coordinates.
\end{corollary}

The following result relies on the analysis of the zero-level set of the potential, which is a minimal (totally geodesic) hypersurface. In many known examples (see Section~\ref{sec:examples}) the zero level set of the potential (if connected) has index one. We study geometric properties of static manifolds with boundary, whose zero-level set of the potential is connected and has index one.

\begin{theorem}\label{thm:ind}
Let $(M,g,V)$ be a compact 3-dimensional static manifold with boundary with $R_g=6\epsilon$, where $\epsilon \in \{-1,0,1\}$. Suppose that $\Sigma=V^{-1}(0)$ is connected, intersects the boundary of $M$ and has index one. Then there exists a positive constant $C$ such that 
\begin{equation}\label{eq:mineq}
\frac12\sum_{i=1}^bH_i|\partial_i\Sigma| \leqslant 2\pi(\gamma+b)+C|\Sigma|\left(\max_\Sigma K_\Sigma-3\epsilon\right),
\end{equation} 
where $\gamma$ and $b$ is the genus and the number of boundary components of $\Sigma$, respectively, $K_\Sigma$ is its Gaussian curvature, $H_i$ is the mean curvature of the connected component of $\partial M$, containing the connected component $\partial_i\Sigma,~i=1,\ldots, b$ of $\partial\Sigma$. Moreover, in the case where $R_g=0,~H_g=2$ equality in~\eqref{eq:mineq} is achieved if, and only if, $(M,g)$ is isometric to the Euclidean unit ball centered at the origin $(\mathbb B^3,\delta)$, $\Sigma$ is a flat unit disk, and $V$ is given by $V(x) = x\cdot v$ for some vector $v \in \mathbb R^3 \setminus \{0\}$.
\end{theorem}

Here and everywhere in the sequel $|M|$ denotes the volume of a compact submanifold $M$ of a Riemannian manifold $(N,h)$ with respect to the induced metric.

\begin{remark}
In the case where $\Sigma$ has a connected component of index one, inequality~\eqref{eq:mineq} holds for that component. If equality in~\eqref{eq:mineq} holds, then the Gauss curvature of $\Sigma$ is constant.
\end{remark}

Let us give some motivation for our next result. Shen in~\cite{shen1997note} and Boucher, Gibbons, and Horowitz in~\cite{boucher1984uniqueness} proved the following theorem.

\begin{theorem}
Let $(M^3,g,V)$ be a compact static manifold. Suppose that $R_g=6$ and $\Sigma=V^{-1}(0)$ is connected. Then $\Sigma$ is a two-sphere and $|\Sigma|\leqslant 4\pi$. Moreover, equality holds if and only if $(M^3,g)$ is isometric to the standard 3-sphere.
\end{theorem}

This theorem was later generalized to allow more boundary components by Ambrozio in~\cite[Proposition 6]{ambrozio2017static}. Very recently, Cruz and Nunes in~\cite{cruz2023static} (see Theorem 2 therein) proved an analog of this theorem for static manifolds with boundary. Our next result generalizes Ambrozio's formula for static manifolds with boundary.

\begin{theorem}\label{thm:ric}
 Let $(M^3,g,V)$ be a compact orientable static manifold with boundary such that $\Sigma=V^{-1}(0)\subset Int~M$. Let $R_g=6\epsilon$, where $\epsilon \in \{-1,0,1\}$. Consider a connected component $\Omega$ of $M\setminus\Sigma$. Let $S_1,\ldots,S_b$ be the connected components of $\partial M\cap \partial\Omega$, $\Sigma_1,\ldots,\Sigma_r$ the connected components of $\Sigma\cap \Omega$, and $\kappa_{i}$ the value of $|\nabla^g V|_g$ on $\Sigma_i$. Then the following identity holds:
\begin{align}\label{eq:tric}
\int_{\Omega}V|\mathring{\Ric_g}|^2_g\,dv_g+\sum_{j=1}^bH_j\left(\frac{H_j}{2}+\epsilon\right)\int_{S_j}V\,&ds_g=\\\nonumber&= sign(V)\sum_{i=1}^r\kappa_i\left(2\pi\chi(\Sigma_i)-\epsilon|\Sigma_i|\right).
\end{align}
Here, $\mathring{\Ric_g}$ stands for the trace-free part of the Ricci tensor. 
\end{theorem}

\begin{remark}\label{rem:highdim}
The proof follows the same steps as the proof of Proposition 6 in~\cite{ambrozio2017static}. However, formula \eqref{eq:tric} also follows from Schoen's Pohozaev-type identity (see~\cite[Proposition 1.4]{schoen1988existence}) by applying it to the vector field $\nabla^gV$ on $\Omega$.
\end{remark}

\begin{remark}
It is not difficult to generalize this theorem to any dimension $n$.
\end{remark}

Let us return to the example of $(\mathbb B^3,\delta,V)$, where $V$ is given by $V(x) = x\cdot v$ for some vector $v \in \mathbb R^3 \setminus \{0\}$. One can notice that $\Ric_\delta(\nu,\nu)+H^2_g/2=2$, which is the first Laplace eigenvalue of the boundary sphere $\partial \mathbb B^3$. This boundary sphere is a \textit{stable cmc-surface} in $\mathbb E^3$ (see Section~\ref{sec:pre} for the corresponding definitions). It also has constant Gauss curvature. Moreover, we see that the zero-level set of $V$ $\Sigma$ intersects $\partial \mathbb B^3$ \textit{only once}, i.e., the set $\Sigma\cap \partial \mathbb B^3$ is connected.  This phenomenon is explained in our next result.

\begin{theorem}\label{thm:scmc}
Let $(M^n,g,V)$ be a compact static manifold with connected boundary and with non-positive cosmological constant. Assume that $\partial M$ is a stable cmc-hypersurface with constant scalar curvature $R_{\partial M}$. Then $\Sigma=V^{-1}(0)$ is empty or it intersects $\partial M$ only once. In the last case $\Sigma$ is connected. If additionally we assume that $H_g/(n-1)=c>0$ is the first non-zero Steklov eigenvalue of $(M,g)$ with $\Ric_g=0$, then $(M,g)$ is isometric to the Euclidean ball of radius $1/c$ centered at the origin and $V(x)=x\cdot v$ for some $v\in \mathbb R^n\setminus \{0\}$. Moreover, when $n=3$, the following inequality holds
\begin{align}\label{eq:isop}
|\partial M|\leqslant \frac{16\pi\left[\dfrac{\gamma+3}{2}\right]}{R_g+\dfrac32 H^2_g-R_{\partial M}},
\end{align}
where $\gamma$ is the genus of $\partial M$. The equality is achieved if, and only if, $\Sigma$ is (up to homotheties) the standard unit sphere or a Bolza surface
$$
B_{\theta} = \{(z,w)\in \mathbb{C}^2\>| \>
w^2= z (z^4+2\cos 2\theta\cdot z^2+1)\} \cup \{(\infty,\,\infty)\},~\theta_1\leqslant \theta \leqslant \pi/2 - \theta_1,
$$
where $\theta_1\approx 0.65$.
\end{theorem}

\begin{remark}
In the case where a boundary component $S$ of $M$ is a Ricci-positive $n$-dimensional manifold, we can guarantee that if $\Sigma$ intersects the boundary, then it does it only once. Indeed, as we know, the hypersurface $S\cap \Sigma$ is totally geodesic in $S$ (see Lemma~\ref{lemma} item $(a.2)$). Then by the \textit{Frankel property} (see~\cite[Section 2]{frankel1966fundamental}), any two totally geodesic (in fact, minimal) hypersurfaces must intersect, which is impossible. Then $\Sigma$ does not intersect $S$, or it intersects $S$ only once. In particular, it holds for the case where $S$ is an $n$-dimensional round sphere.
\end{remark}

\begin{remark}
In fact, in the proof of Theorem~\ref{thm:scmc} we obtain a stronger than~\eqref{eq:isop} inequality: 
$$
|\partial M|\leqslant \frac{2\Lambda_1(\partial M)}{R_g+\dfrac32 H^2_g-R_{\partial M}},
$$
where $\Lambda_1(\partial M)$ is the supremum of the first normalized Laplace eigenvalue of $\partial M$ over all Riemannian metrics (see Section~\ref{sec:pre}). 
\end{remark}

Finally, our last result concerns the non-compact case (for notation, see \textit{Example 6} in Section~\ref{sec:examples}). 

\begin{theorem}\label{thm:noncomp}
Let $(M^n,g,V)$ be a complete (up to the boundary) one-ended asymptotically Schwarzschildean system, which is a static manifold with compact boundary. If $V\neq 0$ on $M\cup \partial M$, $H_g<0$, and $V=const$ on $\partial M$, then $(M,g)$ is isometric to $Sch^n_{-}$ and $V=V_m$. If $n=3$, $\Sigma=V^{-1}(0)\subset Int~M$ is compact and separates the boundary and the end, $\nu(V)=const\neq 0$ on $\Sigma$, $V=const\neq 0$ on $\partial M$, and $H_g>0$, then $(M,g)$ is isometric to $Sch^3_{+}$ and $V=V_m$. Here $V_m$ is the Schwarzschild potential in the isotropic coordinates
$$
V_m(r)=\dfrac{1-\dfrac{m}{2r^{n-2}}}{1+\dfrac{m}{2r^{n-2}}}.
$$   
\end{theorem}

\begin{remark}
Very recently, Raulot proved a similar result, assuming only that the manifold is spin, asymptotically flat, and the boundary is connected (see~\cite[Theorem 6]{raulot2025adm}).
\end{remark}

In this paper, all manifolds are assumed to be orientable.

\subsection{Plan of the paper} The paper is organized as follows. Section~\ref{sec:pre} contains some necessary background information: General properties of static manifolds with boundary, minimal submanifolds and cmc-surfaces and their index, the spectrum of the Laplace operator, and the Steklov spectrum. In Section~\ref{sec:examples} we list some examples of static manifolds with boundary. In Section~\ref{sec:closed} we prove Theorems~\ref{cor:closed1},~\ref{cor:closed2},~\ref{thm:rigidity}, and~\ref{thm:ric}. Here we also prove some corollaries of these theorems and other related propositions. In Section~\ref{sec:fbms} we prove an analog of Theorem~\ref{cor:closed1} for the case where the zero-level set of the potential intersects the boundary. Here we also prove Theorems~\ref{thm:ind} and~\ref{thm:scmc}. Section~\ref{sec:sch} contains the proof of Theorem~\ref{thm:noncomp}. Finally, in the appendix (Section~\ref{appendix}) we consider locally conformally flat static manifolds with boundary.

\subsection{Acknowledgment} This article is an output of a research project implemented as part of the Basic Research Program at the National Research University Higher School of Economics (HSE University). The author thanks the anonymous referee for useful remarks and suggestions. The author thanks Instituto Nacional de Matem\'atica Pura e Aplicada for the hospitality, where part of the research and preparation of this article were conducted. The author is grateful to Lucas Ambrozio for many fruitful discussions and for the invitation to IMPA. 

\section{Preliminaries}\label{sec:pre}

\subsection{General properties of static manifolds with boundary} The following lemma was proved in~\cite[Proposition 1]{cruz2023static}.

 \begin{lemma}\label{lemma} 
 Let  $(M^n,g, V)$ be a static manifold with boundary and  $\Sigma=V^{-1}(0)$ non-empty. Then the following assertions are true.
\begin{itemize}
\item[(a.1)] If $\Sigma$ does not intersect the boundary of $M$, then it is an embedded closed, totally geodesic hypersurface in the interior of $M$, $Int~M$.
\item[(a.2)] If $\Sigma$ intersects the boundary of $M$, then each of its connected components $\Sigma_0$, which has a non-empty intersection with $\partial M$, is an embedded free boundary totally geodesic hypersurface of $M$ and $\partial \Sigma_0$ is a totally geodesic hypersurface of $\partial M$.
\item[(b)] The surface gravity $\kappa:=|\nabla_M V|_g$ is a positive constant on each connected component of $\Sigma$.
\item[(c)] The scalar curvature $R_g$  is constant.
\item[(d)] The boundary $\partial M$ is totally umbilical. Moreover, it has a constant mean curvature $H_g$ on each connected component of $\partial M$ and the mean curvature of $\partial\Sigma_0$ in $\Sigma_0$ (see item (a.2)) is constant on each connected component.
\item[(e)] For any vector $X$ tangent to $\partial M$, where $\nu$ is the unit normal vector field to $\partial M$, one has $\Ric_g(\nu,X)=0$ on $\partial M$.      
\end{itemize}
\end{lemma}

\subsection{Minimal submanifolds and their index} In the previous section, we saw that the zero-level set of the potential is always totally geodesic. Totally geodesic manifolds have zero mean curvature, i.e., they are minimal. Let us briefly recall the definitions of a minimal submanifold and its index. In this section, we rely on the book~\cite{fraser2020geometric}.

\begin{definition}
Let $\varphi\colon \Sigma^k \looparrowright (M,h)$ be an immersion. It is called minimal if it is a critical point of the volume functional
$$
Vol[\varphi]=\int_\Sigma dv_{\varphi^*h}.
$$
If $\partial\Sigma,~\partial M \neq\O$, then we say that an immersion $\varphi$ is minimal with free boundary if it is critical point of the volume functional under the variations leaving the boundary of $\Sigma$ on the boundary of $M$.
\end{definition}

In the sequel, we always identify $\Sigma$ with its image under the immersion. The Euler-Lagrange equation for minimal submanifolds is $H=0$, where $H$ is the mean curvature vector field of $\Sigma$. In the case where $\partial\Sigma,~\partial M \neq\O$, $\Sigma$ is a free boundary minimal subanifold (FBMS for short) if $H=0$ and $\Sigma \perp \partial M$. 

Minimal submanifolds are not necessary points of local minima of the volume functional. In order to define whether a given minimal submanifold is a point of local minimum, we look at the second variation of the volume functional. In our paper, we are interested only in the case where a minimal submanifold is of codimension one and its normal bundle is trivial. In this case, any normal vector field on $\Sigma$ can be written in the form $\phi\nu$, where $\nu$ is the unit normal vector field for $\Sigma$.  The second variation -- which is defined on the normal vector fields -- takes the following form:
\begin{align}\label{eq:2nd}
S(\phi,\phi)=\int_\Sigma \left(|\nabla^g \phi|^2_g-\left(\Ric_h(\nu,\nu)+|B|^2_g\right)\phi^2\right)\, dv_g- \int_{\partial\Sigma}B_{\partial M}(\nu,\nu)\phi^2\,ds_g.
\end{align}
Here $\Ric_h$ is the Ricci tensor of $(M,h)$, $g$ is the metric on $\Sigma$ induced from $h$, and $B$ is the second fundamental form of $\Sigma$ in $(M,h)$. We use the notation $S(\phi,\phi)$ in place of $S(\phi\nu,\phi\nu)$ just for simplicity. In the case where $\Sigma$ does not have boundary, the second integral vanishes. In the case, when $\Sigma$ is 2-dimensional and $M$ is 3-dimensional, the second variation of the volume can be rewritten as
\begin{align}\label{eq:2nd2}
S(\phi,\phi)=\int_\Sigma\left( |\nabla^g \phi|^2_g +\left(-\frac12 \left(R_h+|B|^2_g\right)+K_g\right)\phi^2\right)\, &dv_g\\\nonumber&-\int_{\partial\Sigma}B_{\partial M}(\nu,\nu)\phi^2\,ds_g,
\end{align}
where $R_h$ is the scalar curvature of $(M,h)$ and $K_g$ is the Gaussian curvature of $\Sigma$ with the induced metric. 

One of the main characteristics of a minimal submanifold is its index.

\begin{definition}
The (Morse) index of a minimal submanifold is defined as
$$
\Ind\Sigma=\max\{\dim W~|~S(\phi,\phi)<0~\forall\phi\in W\subset C^\infty(M)\}.
$$
If $\Ind\Sigma=0$, then we call it stable.
\end{definition}

As we will see in Section~\ref{sec:examples}, in many known examples of static manifolds the zero-level sets of the potential have index one. 

\subsection{Constant mean curvature hypersurfaces and their index} In the previous section, we see that a minimal submanifold has zero mean curvature. There exists a generalization of it, which is called a \textit{constant mean curvature hypersurface} (cmc-hypersurface for short). As follows from the name, the mean curvature for these hypersurfaces is of constant length. There is a variational characterization of constant mean hypersurfaces: Similarly to minimal submanifolds, they are critical points of the volume functional under the variations, locally preserving the volume. Clearly, the second variation of the volume functional takes the same form~\eqref{eq:2nd}. However, the functions $\phi$ must satisfy an additional property: $\int_\Sigma \phi\, dv_g=0$. The index is defined in the same way as for minimal submanifolds but with the previous restriction on functions $\phi$. We will only consider closed cmc-hypersurfaces.

\subsection{Spectral problems} \label{sub:spec} In this section we rely on the book~\cite{levitin2023topics}. For our needs, we only consider the spectrum of the Laplacian on a closed Riemannian manifold $(M,g)$. In this case the spectrum consists of eigenvalues, i.e., the numbers $\lambda$ such that the spectral problem $\Delta_gu=-\lambda u$ admits a non-trivial solution $u\in C^\infty(M)$. Then all eigenvalues have finite multiplicities and form the following sequence:
$$
0=\lambda_0(M)<\lambda_1(M)\leqslant \lambda_2(M) \leqslant \ldots \leqslant \lambda_i(M)\leqslant \ldots \nearrow +\infty.
$$
Consider the $i$-th eigenvalue $\lambda_i(M)$ and a corresponding eigenfunction $u_i$. The zero-level set of $u_i$ splits $M$ into connected components, where the function $u_i$ preserves its sign. These components are called \textit{nodal domains}. The celebrated \textit{Courant nodal theorem} states that the function $u_i$ cannot have more than $i+1$ nodal domains. Hence, a zero eigenfunction has exactly one nodal domain (it is also obvious, since any zero eigenfunction is a constant by the maximum principle) and a first eigenfunction cannot have more than two nodal domains. Moreover, another application of the maximum principle to a second eigenfunction implies that in has \textit{exactly two nodal domains}.

In the case where $M$ is a closed surface, the Yang-Yau inequality holds (see~\cite{yang1980eigenvalues,el1983volume,li1982new})
\begin{align}\label{ineq:YY}
\lambda_1(M)|M|\leqslant 8\pi\left[\dfrac{\gamma+3}{2}\right].
\end{align}
In particular, the functional $g\mapsto\lambda_1(M)|M|$ is bounded from above in the space of Riemannian metrics on $M$, which we denote as $\mathcal R(M)$. Then, one can define the following number
$$
\Lambda_1(M):=\sup_{g\in\mathcal R(M)}\lambda_1(M)|M|.
$$
With this definition, the Yang-Yau inequality reads as 
$$
\Lambda_1(M)\leqslant 8\pi\left[\dfrac{\gamma+3}{2}\right].
$$
Sharpness of this inequality was studied in the paper~\cite{karpukhin2019yang}.

Another spectral problem, which we also use in this paper, is the so-called~\textit{Steklov problem}. Let $(M,g)$ be a compact Riemannian manifold with sufficiently smooth boundary. The Steklov problem is the following problem
$$
\begin{cases}
\Delta_gu=0~&\text{in $M$},\\
\dfrac{\partial u}{\partial \nu}=\sigma u~&\text{on $\partial M$}.
\end{cases}
$$
The number $\sigma$ is called a~\textit{Steklov eigenvalue} if there exists a non-zero function $u$, satisfying the Steklov problem. The set of all such numbers is called the~\textit{Steklov spectrum}. It is discrete and has the following form
$$
0=\sigma_0(M)<\sigma_1(M)\leqslant \sigma_2(M) \leqslant \ldots \leqslant \sigma_i(M)\leqslant \ldots \nearrow +\infty.
$$
Here, $\sigma_i(M)$ is the $i$th Steklov eigenvalue. The multiplicity of each Steklov eigenvalue is finite. 

Suppose now that the boundary of $(M,g)$ is smooth. Xia and Wang in~\cite[Theorem 1.1]{wang2009sharp} found an inequality between the first Steklov eigenvalue of $(M,g)$ with non-negative Ricci curvature and the first eigenvalue of $\partial M$ with the induced metric. It reads 
\begin{align}\label{eq:xw}
\sigma_1(M)\leqslant \frac{\sqrt{\lambda_1(\partial M)}}{(n-1)c}\left(\sqrt{\lambda_1(\partial M)}+\sqrt{\lambda_1(\partial M)-(n-1)c^2}\right),
\end{align}
where $n=\dim M$ and $c>0$ is a lower bound on principal curvatures of $\partial M$ in $(M,g)$. Moreover, equality holds if and only if $(M,g)$ is isometric to an $n$-dimensional Euclidean ball of radius $\dfrac{1}{c}$.

\section{Examples of static manifolds with boundary}\label{sec:examples}

In this section, we consider some examples of static manifolds with boundary.

\medskip

\textit{Example 1.} Any bounded domain with totally geodesic boundary on a Ricci-flat manifold is an example of a compact static manifold with boundary with non-vanishing static potential (see~\cite[Proposition 3.3.]{ho2020deformation}). For example, any strip, bounded by parallel hyperplanes or half-spaces in the Euclidean space with a constant static potential, is a static manifold with boundary. However, these examples are not compact. A half of a flat torus of any dimension with a constant function as a potential is an example of a compact static manifold with boundary. Conversely, if the potential of a static manifold with boundary is constant, then it is Ricci-flat and the boundary is totally geodesic. Moreover, if the dimension of such a static manifold with boundary is 3, then it is flat.

\medskip

\textit{Example 2.} The $n$-ball with the Euclidean metric with static potential, given by $V(x)=x\cdot v$ for some $v\in \mathbb R^n$ is an example of a static manifold with connected boundary. The zero-level set of $V$ is an $(n-1)$-dimensional ball, which is a free boundary totally geodesic submanifold. It is well-known that its index is one.

\medskip

\textit{Example 3.} Another example of a compact static manifold with boundary is a spherical cap (or a spherical $n$-ball): A static potential can be chosen as $x_1$, the first coordinate function, when we consider the cap embedded in the Euclidean space $\mathbb E^{n+1}$ with coordinates $x_1,\ldots, x_{n+1}$ (see~\cite[Theorem 4.3]{sheng2024localized}). The zero-level set of the potential is a geodesic $(n-1)$-sphere on $\mathbb S^n$. It meets the boundary of the cap orthogonally. Its index is one (see~\cite[Theorem 5.4]{medvedev2025free}).

\medskip

\textit{Example 4.} This example is similar to the previous two. We consider a hyperbolic $n$-ball. If the $n$-dimensional hyperbolic space is realized as the hyperboloid model, then the ball is centered at the point $(1,0,\ldots,0)$ in the $(n+1)$-dimensional Minkowski space with coordinates $x_0,x_1,\ldots, x_n$. A static potential can be chosen as $x_1$ (see~\cite[Proposition 4.11]{sheng2024localized}). The zero-level set is a geodesic $(n-1)$-ball, which is orthogonal to the boundary sphere of the ball. Its index is one (see again~\cite[Theorem 5.4]{medvedev2025free}).

\medskip

\textit{Example 5.} One can take a spherical $n$-ring, i.e., a domain on the $n$-sphere, bounded by two concentric spherical caps. It is an example of a compact static manifold with two boundary components (we take $x_1$ as a potential once again). The zero-level set of the potential is a spherical $(n-1)$-ring. Its index is one. A hyperbolic $n$-ring is defined in a similar way. The index of the zero-level set of the potential is also one. 

\medskip

\textit{Preliminary Computation.} This computation generalizes \textit{Example 1} in~\cite{cruz2023static} and computations in Section 4.4 in~\cite{sheng2024localized}. Consider a manifold $M^n=\mathbb S^n\times [r_m,+\infty)$ with the spherically symmetric static metric
$$
g_m=V^{-2}_m(r)dr^2+r^2 g_0,
$$
where $g_0$ is the standard metric on the sphere $\mathbb S^n$ and $r_m$ is the largest zero of $V$. Assume that $V$ is a static potential. Let $V$ solve the equation
\begin{equation*}
\frac{\partial V_m}{\partial \nu}{g}-V_mB_{g} = 0 \quad \text{on}\quad \Sigma_r,
\end{equation*}
where $\Sigma_r=\mathbb S^2\times\{r\},~r\in [r_m,+\infty)$. We want $V_m$ to satisfy~\eqref{static2} 
$$
\frac{\partial V_m}{\partial \nu}=\frac{H_g}{n-1}V_m \quad \text{on}\quad \Sigma_r.
$$ 
It is not hard to see that $\Sigma_r$ is umbilical with mean curvature equal $(n-1)V_m/r$. Then 
\begin{equation}\label{eq:bound}
V_m\frac{\partial V_m}{\partial r}=\frac{V^2_m}{r} \Rightarrow \frac{\partial V_m}{\partial r}=\frac{V_m}{r}.
\end{equation}
Let $V_m(r)=\sqrt{1+\epsilon r^2-\dfrac{2m}{r^{n-2}}}$, where $\epsilon \in\{0,1\}$. The triple $(M^n,g_m,V_m)$ is called the \textit{Schwarzschild} ($Sch^n$) or \textit{anti-de Sitter-Schwarzschild space} ($AdS$-$Sch^n$) for $\epsilon=0,1$, respectively. The solution to equation~\eqref{eq:bound} in these cases yields some number $r_{ps}=\left(nm\right)^{\frac{1}{n-2}}$. The sphere $\Sigma_{r_{ps}}$ is known as a space-like slice of the \textit{photon sphere} $\Sigma_{r_{ps}}\times \mathbb R$ in the corresponding space-times (see definitions in~\cite{claudel2001geometry, perlick2005totally}). For simplicity we call $\Sigma_{r_{ps}}$ just the \textit{photon sphere}. When $\epsilon=-1$, we obtain the \textit{de Sitter-Schwarzschild space} ($dS$-$Sch^n$). For it $r\in (r_-(m),r_+(m))$, where $r_\pm$ are the two positive roots of $V_m$. The boundary sphere $\mathbb S^2\times\{r_-(m)\}$ is called the \textit{black hole horizon} and $\mathbb S^2\times\{r_+(m)\}$ is called the \textit{cosmological horizon}. The solution to equation~\eqref{eq:bound} is also $r_{ps}=\left(nm\right)^{\frac{1}{n-2}}$, provided that $m\in \left(0,\dfrac{1}{n}\left(1-\dfrac{2}{n}\right)^{\frac{2}{n-2}}\right)$. (See~\cite[Example 5]{claudel2001geometry},~\cite[Formula (11)]{virbhadra2002gravitational}, and~\cite[Formula (2.19)]{riojas2023photon} for a justification that $r_{ps}=\left(nm\right)^{\frac{1}{n-2}}$ for $Sch^n$, $AdS$-$Sch^n$, and $dS$-$Sch^n$.)

\begin{remark}
The fact that $\Sigma_{r_{ps}}$ is nothing but the photon sphere is not a coincidence. To the best of our knowledge, this fact was first observed in~\cite[Remark 1]{cruz2023static}. We also explain this phenomenon in Lemma~\ref{thm:photon} below, using the more general notion of the \textit{quasi-local photon surface}. 
\end{remark}

\medskip

\textit{Example 6.} Consider the $n$-dimensional Schwarzschild space $Sch^n$ and the photon sphere in it. This sphere splits the Schwarzschild space into two parts: the compact one and the non-compact one, the one which contains the \textit{end}. The non-compact part is an example of a static manifold with boundary. We denote it $Sch^n_-$. The static potential does not vanish on it. If we take the compact part and perform the \textit{doubling procedure} by reflection across the horizon, we obtain a compact static manifold with boundary (see~\cite[Example 1]{cruz2023static}). The zero-level set of the potential is the horizon. A simple use of formula~\eqref{eq:2nd} shows that it is stable. We can also perform the doubling for the Schwarzschild space. Consider the domain on the doubled Schwarzschild space $\widetilde{Sch^n}$, bounded by the image of the photon sphere, the one that contains $Sch^n$. It is another example of a non-compact static manifold with boundary. It also has one end (the same as the end of $Sch^n$), but the static potential vanishes on it. We denote it $Sch^n_+$. The zero-level set of the potential is the horizon. It is stable. 

\medskip

\textit{Example 7.} One can do the same as we did in the previous example for the anti-de Sitter-Schwarzschild space. Denote the non-compact static manifolds with boundary by $AdS$-$Sch^n_-$ and $AdS$-$Sch^n_+$, respectively. The zero-level set of the potential, which is exactly the horizon, is also stable. The doubling of the de Sitter-Schwarzschild space $\widetilde{dS-Sch^n}$ is compact. Consider the photon sphere and its image after the doubling. They split $\widetilde{dS-Sch^n}$ into two connected parts: the one that contains the cosmological horizon, which we denote as $dS$-$Sch^n_+$, and the one that contains the black hole horizon, which we denote as $dS$-$Sch^n_-$. They are compact static manifolds with boundary (see~Theorem~\ref{thm:sheng}). 

\medskip

\textit{Example 8.} Consider the doubled $n$-dimensional Schwarzschild space $\widetilde{Sch^n}$. In the \textit{isotropic coordinates} it is conformally equivalent to the Euclidean space $\mathbb E^n$ without the origin, which corresponds to the \textit{singularity}. Consider a hyperplane passing through the origin in $\mathbb E^n$. It is not hard to see that this hyperplane is totally geodesic in $\widetilde{Sch^n}$. It splits $\widetilde{Sch^n}$ into two isometric pieces, each of which we denote as $1/2Sch^n$ and call the \textit{hemi-Schwarschild space}. It was shown in~\cite[Proposition 4.24]{sheng2024localized}, that $1/2Sch^n$ with the Schwarzschild static potential is a non-compact static manifold with boundary. The zero-level set of the potential is half of the horizon, which is stable. 

\medskip

\textit{Example 9.} Similarly to the previous example, we define the space $1/2AdS$-$Sch^n$. Consider $dS$-$Sch^n$. It is also conformally equivalent to a domain in $\mathbb E^n$, bounded by two concentric spheres, centered at the origin. A hyperplane passing through the origin is totally geodesic in $dS$-$Sch^n$. Consider the image of this hyperplane after the doubling. The hyperplane and its image split $\widetilde{dS-Sch^n}$ into two isometric pieces, each of which we denote $1/2dS$-$Sch^n$. The spaces $1/2AdS$-$Sch^n$ and $1/2dS$-$Sch^n$ are static manifolds with boundary (see~Theorem~\ref{thm:sheng}).

\medskip

\textit{Example 10.} The last example is the \textit{Nariai manifold with boundary}: Consider the \textit{Nariai solution} $Nar(M^{n-1}):=\left(\mathbb R\times M^{n-1},g=\dfrac{1}{n}dr^2+\dfrac{n-2}{n}g_M,V(r)=\sin(r)\right)$, $n\geqslant 3$. Here $(M^{n-1},g_M)$ is a simply connected Riemannian manifold such that $\Ric_{g_M}=(n-2)g_{M}$. Clearly, $\{r\}\times M$ is totally geodesic for any $r$ and $\nu=\sqrt{n}\dfrac{\partial}{\partial r}$. Hence, on $\left\{\dfrac{\pi}{2}+\pi k\right\}\times M$, with $k\in\mathbb Z$ the second equation in~\eqref{static} is satisfied. Then we obtain compact static manifolds with boundary $Nar_{k,l}(M^{n-1}):=\left(\left[\dfrac{\pi}{2}+\pi k, \dfrac{\pi}{2}+\pi l\right]\times M, g,V(r)=\sin (r)\right)$, where $k<l$ are any two integer numbers. We call it a \textit{compact Nariai manifold with boundary}. We can also obtain \textit{non-compact Nariai static manifolds with boundary}: 
 \begin{align*}
 Nar_{-k}(M^{n-1}):=\left(\left(-\infty, \dfrac{\pi}{2}+\pi k\right]\times M, g,V(r)=\sin (r)\right)~\text{ and}
 \end{align*}
 \begin{align*}
 Nar_{k+}(M^{n-1}):=\left(\left[\dfrac{\pi}{2}+\pi k,+\infty\right)\times M, g,V(r)=\sin (r)\right),~\text{ where }~k\in \mathbb Z.
 \end{align*}
 The zero-level set of the potential, which is possibly disconnected, is stable. Finally, when $M^{n-1}=\mathbb S^{n-1}$, the space $Nar(\mathbb S^{n-1})$ is conformally equivalent to $\mathbb E^n$. Then considering a hyperplane passing through the origin as in \textit{Example 8}, we obtain the space $1/2Nar(\mathbb S^{n-1})$. It is also a static manifold with boundary (see~Theorem~\ref{thm:sheng}).

\section{Compact case: The zero-level set of the potential does not intersect the boundary}\label{sec:closed}

 Let us first focus on the case where $M$ is compact and the zero-level set of the potential $\Sigma$ does not intersect $\partial M$. 
\begin{lemma}\label{lemma:main}
Let $(M^n,g,V)$ be a compact static manifold with boundary with cosmological constant $\Lambda \leqslant 0$. Assume that $V^{-1}(0)=\Sigma\subset Int~M$. Then any connected component of $M\setminus \Sigma$ contains a boundary component of $M$.
\end{lemma}

\begin{proof}
We consider the case where $\Lambda<0$. The case where $\Lambda=0$ is similar. Suppose that there exists a connected component $\Omega$ of $M\setminus \Sigma$, which does not contain a boundary component of $M$. Clearly, $V$ does not change its sign in $\Omega$. Let, for example, $V>0$. Then it follows from equation $\Delta_gV=-\Lambda V$ (see~\eqref{static1}) that $V$ is subharmonic. By the weak maximum principle, $V$ can only achieve its maximum at the boundary of $\Omega$. However, $V=0$ on $\partial\Omega$. We arrive at a contradiction. The case where $V<0$ is similar.
\end{proof}

Observe that if $\Lambda>0$, then it can happen that a connected component $\Omega$ of $M\setminus \Sigma$ contains no boundary component of $M$. In this case $(\Omega,g,V|_\Omega)$ is a \textit{static triple}, i.e., $V|_\Omega\neq0$ and vanishes exactly on $\partial \Omega$. 
\begin{corollary}
Let $(M^n,g,V)$ be a compact static manifold with boundary with non-positive cosmological constant such that $V^{-1}(0)=\Sigma\subset Int~M$. Then the number of connected components of $M\setminus \Sigma$ is not greater than the number of connected components of $\partial M$.
\end{corollary}

\begin{proof}
Otherwise, by the Dirichlet (pigeonhole) principle, there exists a connected component of $M\setminus \Sigma$, which does not contain a boundary component of $M$. It contradicts Lemma~\ref{lemma:main}.
\end{proof}

Notice that, however, a connected component of $M\setminus \Sigma$ can contain \textit{many} boundary components of $M$.

As a simple corollary of the previous observations, we get Theorems~\ref{cor:closed1} and~\ref{cor:closed2}.

We proceed with the proof of Theorem~\ref{thm:ric}, which is an analog of Proposition 6 in~\cite{ambrozio2017static}. 

\begin{proof}[Proof of Theorem~\ref{thm:ric}]
We prove this theorem for the case where $V>0$. The case where $V<0$, is ruled out by passing to $-V$. 

We use Shen's identity (see~\cite[Formula (12)]{shen1997note}), which reads
\begin{equation}\label{eq:shen}
  \mydiv_g \left(\frac{1}{V}d\left(|\nabla^g V|^2_g+\epsilon V^2\right)\right) = 2V|\mathring{\Ric_g}|^2_g.
\end{equation}
Consider the following vector field
\begin{equation*}
 X=\frac{1}{V}\nabla^g\left(|\nabla^g V|^2_g+\epsilon V^2\right).
\end{equation*} 
Shen's identity implies that 
\begin{align}\label{eq:divX}
\mydiv_g X = 2V|\mathring{\Ric_g}|^2_g \quad \text{on $\Omega$.}
\end{align}
It follows from the definitions of Hessian and the first equation in \eqref{static}, that 
\begin{align}\label{eq:X}
X = 2(\Ric_g(\nabla^g V,\cdot) - 2\epsilon\nabla^g V).
\end{align}
Indeed, for any vector field $Y$ on $\Omega$ by definition of gradient, we have
$$
\left\langle \frac{1}{V}\nabla^g\left(|\nabla^g V|^2_g+\epsilon V^2\right),Y\right\rangle_g = \frac{1}{V}\left(Y(|\nabla^g V|^2_g)+\epsilon Y(V^2)\right).
$$
Differentiating along $Y$ implies
\begin{align*}
\left\langle \frac{1}{V}\nabla^g\left(|\nabla^g V|^2_g+\epsilon V^2\right),Y\right\rangle_g =\frac{2}{V}\langle \nabla_Y\nabla^gV,\nabla^gV\rangle_g+2\epsilon Y(V)\\=2\Hess_gV(\nabla^gV,Y)+2\epsilon Y(V),
\end{align*}
where we used the definition of the Hessian. Using the first equation in \eqref{static} and the fact that $\Delta_gV=-3\epsilon V$ (see the second equation in~ \eqref{static1}), we then get
$$
\left\langle \frac{1}{V}\nabla^g\left(|\nabla^g V|^2+\epsilon V^2\right),Y\right\rangle_g=2V\Ric_g(\nabla^gV,Y)-4\epsilon Y(V),
$$
i.e.,
$$
\langle X,Y\rangle_g=\langle 2(\Ric_g(\nabla^g V,\cdot) - 2\epsilon\nabla^g V),Y\rangle_g
$$
for any vector field $Y$ on $\Omega$, which proves formula~\eqref{eq:X}.

Now we integrate \eqref{eq:shen} against $\Omega$:
\begin{align}\label{eq:begin}
 \int_{\Omega}V|\mathring{\Ric_g}|^2_g\,dv_g = \frac{1}{2}\int_{\Omega}\mydiv_g X \,dv_g =\frac{1}{2}\int_{\partial \Omega} \left\langle X,\nu\right\rangle_g\,ds_g.
\end{align}
Here we used~\eqref{eq:divX} and the divergence theorem. Recall that $\nu$ is the outward unit normal vector field to $\partial\Omega$. Further, using~\eqref{eq:X}, we obtain
\begin{align}\label{eq:nuX}
\left\langle X,\nu\right\rangle_g=2(\Ric_g(\nabla^gV,\nu)-2\epsilon\nu(V)).
\end{align}
 Observe that $\nu$ can be expressed as $-\displaystyle\frac{\nabla^g V}{|\nabla^g V|_g}$ along $\Sigma\cap\partial\Omega$. Hence,
\begin{equation}\label{eq:Xnu}
\left\langle X,\nu\right\rangle_g=2\left(2\epsilon- \Ric_g\left(\frac{\nabla^g V}{|\nabla^g V|_g},\frac{\nabla^g V}{|\nabla^g V|_g}\right)\right)|\nabla^g V|_g \quad \text{on $\Sigma\cap\partial\Omega$}.
 \end{equation}
By the contracted Gauss equation, we have
\begin{equation}\label{eq:cgauss}
 K_{\Sigma_i} =3\epsilon-\Ric_g\left(\frac{\nabla^g V}{|\nabla^g V|_g},\frac{\nabla^g V}{|\nabla^g V|_g}\right)~\text{on a connected component $\Sigma_i$ of $\Sigma \cap \partial\Omega$}.
\end{equation}
Here, $K_{\Sigma_i}$ is the Gauss curvature of $\Sigma_i$ with the induced metric. Plugging \eqref{eq:cgauss} into \eqref{eq:Xnu}, we get
\begin{equation}\label{eq:Xnu2}
\left\langle X,\nu\right\rangle_g= 2(K_{\Sigma_i}-\epsilon)\kappa_i~\text{on a connected component $\Sigma_i$ of $\Sigma \cap \partial\Omega$},
\end{equation}
where we also used the fact that $|\nabla^g V|_g=\kappa_i=const$ on $\Sigma_i$ by item $(b)$ in Lemma~\ref{lemma}. 

On the other hand, on any connected component $S_j$ of $\partial\Omega\cap\partial M$ we have $\nabla^gV=\nabla^{S_j}V+V\nu$, where $\nabla^{S_j}$ is the gradient in the induced metric on $S_j$. Then along $\partial\Omega\cap\partial M$  formula \eqref{eq:nuX} yields
\begin{equation}\label{eq:dXnu}
\left\langle X,\nu\right\rangle_g=2(\Ric_g(\nabla^{S_j}V,\nu)+V\Ric_g(\nu,\nu)-2\epsilon\nu(V))=2(V\Ric_g(\nu,\nu)-2\epsilon\nu(V)),
\end{equation}
since $\Ric_g(X,\nu)=0$ for any $X$ tangent to $S_j$ by item $(e)$ in Lemma~\ref{lemma}. Further, it is easy to see that on $S_j$
\begin{align}\label{eq:comp}
V\Ric_g(\nu,\nu)=-\Delta_{S_j}V-\frac{H^2_{j}}{2}V,
\end{align}  
where $\Delta_{S_j}$ is the Laplacian of the induced on $S_j$ metric. Indeed, we plug the well-known formula
$$
\Delta_gV=\Delta_{S_j}V+H_{S_j}\nu(V)+\Hess_gV(\nu,\nu) \quad \text{on $S_j$}
$$
into the first equation in~\eqref{static}, computed on $(\nu,\nu)$
$$
V\Ric_g(\nu,\nu)=\Hess_gV(\nu,\nu)-\Delta_gV 
$$
which after a simplification yields~\eqref{eq:comp}. Then substituting \eqref{eq:comp} and using the second equation in \eqref{static2} in \eqref{eq:dXnu}, we get
\begin{equation}\label{eq:HV}
\left\langle X,\nu\right\rangle_g=-2\left(\Delta_{S_j}V+\frac{H^2_{j}}{2}V+\epsilon H_jV\right) \quad \text{on $S_j$}.
\end{equation}
Now we return to formula \eqref{eq:begin} and use formulae \eqref{eq:Xnu2} and \eqref{eq:HV}:
\begin{align}\label{eq:almost}
 \int_{\Omega}&V|\mathring{\Ric_g}|^2_g\,dv_g =\frac{1}{2}\int_{\partial \Omega} \left\langle X,\nu\right\rangle_g\,ds_g\\ \nonumber &=\sum_{i=1}^r\int_{\Sigma_i}(K_{\Sigma_i}-\epsilon)\kappa_i\,ds_g+\sum_{j=1}^b\int_{S_j}\left(-\Delta_{S_j}V-\frac{H^2_{j}}{2}V-\epsilon H_jV\right)\,ds_g.
\end{align}
Finally, observing that by the Gauss-Bonnet theorem $\int_{\Sigma_i}K_{\Sigma_i}\,ds_g=2\pi\chi(\Sigma_i)$, by the divergence theorem $\int_{S_j}\Delta_{S_j}V\,ds_g=0$, since $S_j$ is closed, and $H_j=const$ by item $(d)$ in Lemma~\ref{lemma}, we get that formula~\eqref{eq:almost} implies \eqref{eq:tric}.
\end{proof}

Consider some curious corollaries of Theorem~\ref{thm:ric}.

\begin{corollary}\label{cor:tric_1}
Let $(M^3,g,V)$ be a compact static manifold with boundary such that $\Sigma=V^{-1}(0)\subset Int~M$. Assume that the cosmological constant of $(M,g,V)$ is positive. Let $\Sigma_1,\ldots,\Sigma_r$ be the connected components of $\Sigma$ and $H_j$ the mean curvature of the connected component $S_j$ of $\partial M$. Then either one of $\Sigma_i$ is a topological sphere or $H_j<0$ for some $j$.
\end{corollary}

\begin{proof}
Applying a metric rescaling, without loss of generality, one can assume that $R_g=6$. If $\chi(\Sigma_i) \leqslant 0$ for all $i$ and $H_j \geqslant 0$ for all $j$, then in formula~\eqref{eq:tric} with $\epsilon=1$ we see that the sign of the left-hand side is opposite to the sign of the right-hand side or it is zero, while the right-hand side is never zero, which is impossible.
\end{proof}

\begin{corollary}\label{cor:tric_2}
Let $(M^3,g,V)$ be a compact static manifold with boundary with $R_g=6\epsilon$ and such that $\Sigma=V^{-1}(0)\subset Int~M$. Assume that $M\setminus \Sigma$ has exactly two connected components. Then
$$
\int_{M}V|\mathring{\Ric_g}|^2_g\,dv_g=-\sum_{j=1}^bH_j\left(\frac{H_j}{2}+\epsilon\right)\int_{S_j}V\,ds_g,
$$
where $b$ is the total number of connected components $S_j$ of $\partial M$. 

\end{corollary}

\begin{proof}
The first formula follows immediately from formula~\eqref{eq:tric} on the connected components of $M\setminus \Sigma$ (after summation).  
\end{proof}

Finally, we observe that an analog of Proposition 15 in~\cite{ambrozio2017static} also holds true in the following sense.

\begin{proposition} 
 Let $(M^3,g,V)$ be a compact static manifold with boundary and $\Sigma=V^{-1}(0)\subset Int~M$. Then the inclusion map $i\colon\Sigma \to M$ induces an injective map $i_{*}\colon \pi_{1}(\Sigma) \rightarrow \pi_{1}(M)$ between the fundamental groups. In particular, if $M$ is simply connected, then $\Sigma$ is a collection of spheres.
\end{proposition}

The proof of this proposition is the same as the proof of Proposition 15 in~\cite{ambrozio2017static}. 

We conclude this section with the proof of Theorem~\ref{thm:rigidity}.

\begin{proof}[Proof of Theorem~\ref{thm:rigidity}]
Consider $M\setminus \Sigma$. Since $\Sigma \subset Int(M)$ is connected, $M\setminus \Sigma$ has exactly two boundary components. Let $\Omega$ denote a connected component of $M\setminus \Sigma$, containing a boundary component of $M$. By Lemma~\ref{lemma:main}, if $\epsilon=0$, then any connected component of $M\setminus \Sigma$ contains a boundary component of $M$. But it is not necessarily the case if $\epsilon=1$. 

Let $\cup_{i=0}^kS_i=\partial\Omega\setminus \Sigma$. We want to show that $k=0$. For this aim we use the construction of the \textit{singular Einstein manifold associated to $\Omega$} (see~\cite[Section 6]{ambrozio2017static}) and apply the \textit{Frankel argument} (see~\cite{frankel1961manifolds,frankel1966fundamental,fraser2014compactness}). 

Let us recall the construction of the singular Einstein manifolds. We refer the interested readers to~\cite[Section 6]{ambrozio2017static}, where they can find the references for the statements that we mention in this paragraph. Let $\Omega$ be a connected component of $M\setminus \Sigma$. Let $\partial \Omega=\Sigma\cup \left (\cup_{i=0}^kS_i\right)$, where each $S_i$ is a connected hypersurface. Let $U$ be a sufficiently small tubular neighbourhood of $\Sigma$ diffeomorphic to $[0,1)\times \Sigma$. The singular Einstein manifold with boundary $N^{n+1}$ associated to $\Omega$ is defined as the quotient space $N^{n+1}=\mathbb S^1\times (\Omega\setminus\Sigma)\,\sqcup\,(\mathbb B^2_1\times \Sigma)/\sim,$
where $\mathbb S^1$ and $\mathbb B^2_1$  are the unit circle and the open unit ball centered at the origin of $\mathbb{R}^2$, respectively, and $\sim$ is the equivalence relation, which identifies $(\theta,p)\in \mathbb S^1\times (U\setminus\Sigma)$ with $(r\cos\theta, r\sin \theta, x)\in \mathbb B^2_1\times \Sigma$ if $p=(r,x)$. Notice that $\partial N=\mathbb S^1\times (\partial\Omega\setminus \Sigma)$. Since $V$ does not vanish on $\Omega\setminus \Sigma$, we can define the Riemannian metric $h=V^2d\theta^2+g$ on $N^{n+1}\setminus \Sigma$. This Riemannian metric is singular along $\Sigma$. In fact, it has, up to perturbation, a conical behavior with angle $k=|\nabla_{\Sigma}V|>0$ (it is a constant by $(d)$ in Lemma~\ref{lemma}) in the directions transverse to $\Sigma$. Moreover, $h$ is indeed Einstein: Its Ricci curvature satisfies $\Ric_h=\epsilon nh$, if the scalar curvature of $g$ on $\Omega$ is $\epsilon n(n-1)$. Finally, rescaling $V$ is necessary, we can always assume that $|\nabla_\Sigma V|=1$. Then the metric $h$ extends smoothly on $\mathbb S^1\times \Sigma$. 

Further, it is not hard to see that the boundary components of $N$, which are diffeomorphic to $\mathbb S^1\times S_i$, are totally umbilical with principal curvature $\dfrac{H_{S_i}}{n-1}$ (see~\cite[Section 4]{cruz2023static}). Then the mean curvature of each boundary component of $N$ is equal to $\dfrac{n}{n-1}H_{S_i}$. Suppose that $N$ has at least two connected boundary components $\partial_1N$ and $\partial_2N$. Then there exists a minimizing geodesic $\gamma$ that realizes the distance between $\partial_1N$ and $\partial_2N$. Assume that $\gamma$ is parametrized by arc length $t\in [0,l]$, where $l$ is the distance between  $\partial_1N$ and $\partial_2N$. Let $p=\gamma(0)\in \partial_1N$. Notice that $\gamma$ meets both $\partial_1N$ and $\partial_2N$ orthogonally (otherwise, $\gamma$ could be shorten). Consider the orthonormal basis $\{\dot\gamma, e_1\ldots,e_n\}$, where $\{e_1,\ldots,e_n\}$ is an orthonormal basis in $T_p\partial_1N$. Abusing notations, let $\dot\gamma, e_1,\ldots,e_n$ denote the fields obtained by the parallel transport of this basis along $\gamma$. Since $\gamma$ is minimizing, the second variation of the length of $\gamma$ along $e_i$, which is denoted as $\delta^2L(e_i,e_i)$, is non-negative for any $i=1,\ldots,n$. Then by Synge's formula (see e.g.~\cite[p.70]{frankel1966fundamental})
\begin{align}\label{synge}
0\leqslant \sum_{i=1}^n\delta^2L(e_i,e_i)=&-\int_0^l\Ric_h(\dot\gamma,\dot\gamma)dt-H_{\partial_1N}-H_{\partial_2N}\\ \nonumber&=-\epsilon nl-\dfrac{n}{n-1}H_{S_1}-\dfrac{n}{n-1}H_{S_2},
\end{align}
since $\Ric_h=\epsilon nh$. Then the right-hand side of~\eqref{synge} is negative  in the case $(i)$, when $\epsilon=1$ and $H_{S_i}\geqslant 0$. We arrive at a contradiction. Then $\partial N$ is connected and, hence, $\Omega$ contains a unique boundary component of $M$. If the remaining connected component of $M\setminus \Sigma$ does not contain boundary components of $M$, then we conclude that $M$ has one boundary component. Otherwise, the same arguments as above, imply that the remaining boundary component also contain exactly one boundary component of $M$. Then $M$ has two boundary components. 

In the case $(ii)$, when $\epsilon=0,~H_{S_i}>0$, we similarly obtain that the right-hand side of~\eqref{synge} is negative and we conclude that $\Omega$ contains a unique boundary component $S$ of $M$. Unlike the case $(i)$, we have that both connected components of $M\setminus \Sigma$ contain one boundary component of $M$. Hence, $M$ has two boundary components. 
  
Pass to the second part of the theorem. As we have seen above in~\eqref{eq:comp}, the following formula holds
\begin{equation}\label{eq1}
\Delta_SV=-\left(\Ric_g(\nu,\nu)+\dfrac{H_S^2}{n-1}\right)V,
\end{equation}
where $S$ is a connected component of $\partial M$. But $V=const$ on $S$ by assumption. Hence, $\Ric_g(\nu,\nu)+\dfrac{H_S^2}{n-1}=0$ and by the contracted Gauss equation 
\begin{align}\label{eq:scal}
R_S=R_g+\dfrac{n}{n-1}H_S^2>0
\end{align}
in both cases $(i)$ and $(ii)$. 

Finally, when $n=3$, then formula~\eqref{eq:scal} implies that the Gauss curvature of $S$ is a positive constant. Then by the Gauss-Bonnet theorem, $S$ is a sphere of constant Gauss curvature. Then Hopf's theorem (see~\cite[Section 3]{hopf2001clifford}) implies that it is a round sphere.
\end{proof}

\section{Compact case: The zero-level set of the potential intersects the boundary}\label{sec:fbms}

Now we move to the case where $M$ is compact and the zero-level set of the potential $\Sigma$ has a non-empty intersection with $\partial M$. The following proposition can be considered as an analog of Theorem~\ref{cor:closed1} for the case where $\Sigma$ intersects the boundary of $M$.

\begin{theorem}\label{thm:bd}
There are no static manifolds with boundary with non-positive cosmological constant $(M^n,g,V)$, satisfying the following properties: 
\begin{itemize}
\item $\Sigma=V^{-1}(0)$ contains a connected component $\Sigma'$, which intersects one connected component $S$ of $\partial M$ and nothing else;\\
\item $R_{S}\geqslant R_g+\displaystyle\frac{n}{n-1}H^2_S$;\\
\item the domain, bounded by $\Sigma'$ and $S$ is compact and contains no other components of $\Sigma$ (i.e., $V$ does not change its sign on this domain)
\end{itemize}
Particularly, if $M$ is compact and $\Sigma$ is connected and intersects a boundary component $S$ of $M$ with $R_{S}\geqslant R_g+\displaystyle\frac{n}{n-1}H^2_S$, then it necessarily intersects another boundary component of $M$, which does not satisfy this pinching condition. In this case $\Sigma$ has more than one boundary component.
\end{theorem}

\begin{proof}
Let $\Omega$ be the (Lipschitz) domain, bounded by $\Sigma'$ and $S$. Without loss of generality, we can assume that $V>0$ on $\Omega$. Since $(M^n,g,V)$ has non-positive cosmological constant, then $V$ is subharmonic. By the weak maximum principle, $V$ can only achieve its maximum in $\Sigma'$ or in $\partial\Omega\cap S$. $V$ cannot achieve its maximum on $\Sigma'$ since $V=0$ on it. Then it achieves its maximum at an interior point of $\partial\Omega\cap S$. Consider the restriction of $V$ on $\partial\Omega\cap S$. It follows from~\eqref{eq:comp} that $V$ satisfies the equation 
$$
\Delta_SV=-\left(\Ric_g(\nu,\nu)+\dfrac{H_S^2}{n-1}\right)V=\frac12\left(R_{S}- R_g-\displaystyle\frac{n}{n-1}H^2_S\right)V,
$$
where we applied the contracted Gauss equation. Then $V$ is subharmonic in $\partial\Omega\cap S$. Again, by the weak maximum principle, $V$ can only achieve its maximum at the boundary of $\partial\Omega\cap S$, which is impossible since $V=0$ on it. A contradiction.
\end{proof}

\begin{remark}
For example, if $R_g=0$ and the boundary of $M$ satisfies $H_S=0,~R_S\geqslant 0$ (as in \textit{Example 8}), then $\Sigma$ and $S$ cannot bound a bounded domain. 
\end{remark}

Recall that according to Lemma~\ref{lemma}, $\Sigma$ is a free boundary totally geodesic hypersurface and the mean curvature of $\partial\Sigma$ in $\Sigma$ is locally constant. Moreover, in the case where $n=3,~R_g=0$ and $H_g=2$, the geodesic curvature of $\partial\Sigma$ in $\Sigma$ is equal to 1. Cruz and Nunes proved in~\cite[Theorem 2]{cruz2023static}, that in the above setting if we additionally assume that $\Sigma$ is connected, then $\Sigma$ is a free boundary totally geodesic 2-disk with area $|\Sigma|\leqslant \pi$. The equality case implies that $\Sigma$ is a flat unit disk and $(M,g)$ is the Euclidean unit ball $(\mathbb B^3,\delta)$. As we saw in \textit{Example 2}, the index of a flat disk in $(\mathbb B^3,\delta)$ is one. We pass to the proof of Theorem~\ref{thm:ind}, which characterizes free boundary zero-level sets of index one. 

\begin{proof}[Proof of Theorem~\ref{thm:ind}]
The second variation of volume of $\Sigma$ along the vector field $\phi\nu$ is given by formula~\eqref{eq:2nd2}, which in our case takes the following form
$$
S(\phi,\phi)=\int_\Sigma |\nabla^g \phi|^2_g\,dv_g-\frac12\int_\Sigma R_g\phi^2\,dv_g+\int_\Sigma K_\Sigma\phi^2\, dv_g-\int_{\partial\Sigma}B_{\partial M}(\nu,\nu)\phi^2\,ds_g.
$$
Here we used that $\Sigma$ is totally geodesic (see Lemma~\ref{lemma} $(a.2)$), i.e., $B=0$. Recall that by Lemma~\ref{lemma} $(d)$ $\partial M$ is umbilical, i.e., $B_{\partial M}(\nu,\nu)=\dfrac{H_i}{2}$ along $\partial_i\Sigma$. Since the index of $\Sigma$ equals one, there exists a function $\phi\in C^\infty(\Sigma)$, which can be chosen positive, such that $S(\phi,\phi)<0$. It is well-known (see~\cite{gabard2006representation}), that $\Sigma$ admits a proper branched conformal covering $u$ over $(\mathbb B^2,\delta)$ of degree $\deg(u) \leqslant \gamma+b$. Let $x_1,x_2$ be coordinate functions on $\mathbb B^2$ and $u_j=u\circ x_j,~j=1,2$. In order to prove inequality~\eqref{eq:mineq} we use the \textit{Hersch trick}. It implies that $u_1,u_2$ can be chosen $L^2(\partial\Sigma,ds_g)$-orthogonal to $\phi$. Then
$$
S(u_j,u_j)=\int_\Sigma |\nabla^g u_j|^2_g\,dv_g-\frac12\int_\Sigma R_gu_j^2\,dv_g+\int_\Sigma K_\Sigma u_j^2\, dv_g-\int_{\partial\Sigma}\frac{H_g}{2}u_j^2\,ds_g\geqslant 0,~j=1,2.
$$
Summing these two inequalities, we get
\begin{equation*}\label{ineq:ax}
\int_\Sigma |\nabla^g u|^2_g\,dv_g-\frac12\int_\Sigma R_g\sum_ju_j^2\,dv_g+\int_\Sigma K_\Sigma \sum_ju_j^2\, dv_g-\int_{\partial\Sigma}\frac{H_g}{2}\,ds_g\geqslant 0,
\end{equation*}
where we have used that $\sum_ju_j^2=1$ on $\partial\Sigma$. In a standard way, we observe that $\int_\Sigma |\nabla^g u|^2_g\,dv_g=2\pi \deg(u)\leqslant 2\pi(\gamma+b)$. Set $\int_\Sigma\sum_ju_j^2\, dv_g=C>0$ and observe that $\int_\Sigma K_\Sigma \sum_ju_j^2\, dv_g\leqslant C\max_\Sigma K_\Sigma$. Then we get~\eqref{eq:mineq}. 

In the case, when $R_g=0,~H_g=2$ inequality~\eqref{eq:mineq} reads
$$
|\partial\Sigma| \leqslant 2\pi(\gamma+b)+C\left(\max_\Sigma K_\Sigma\right).
$$
In the equality case we obtain that $K_\Sigma=const$. Then by the Gauss-Bonnet theorem, we get
$$
K_\Sigma|\Sigma|+|\partial\Sigma|=2\pi\chi(\Sigma),
$$
since the geodesic curvature of $\partial\Sigma$ in $\Sigma$ equals 1. Then we have
$$
2\pi\chi(\Sigma)-K_\Sigma|\Sigma|=2\pi(\gamma+b)+CK_\Sigma|\Sigma|.
$$
Moreover, by~\cite[Theorem 2]{cruz2023static}, $\Sigma$ is diffeomorphic to a disk. Then $\gamma=0,~b=1$, which implies
$$
-K_\Sigma|\Sigma|=CK_\Sigma|\Sigma|.
$$
This is possible if, and only if, $K_\Sigma=0$, i.e., $\Sigma$ is the flat disk and the geodesic curvature of the boundary is 1. Then $\Sigma$ is isometric to the flat unit disk. The area of $\Sigma$ is equal to $\pi$. By the rigidity case in~\cite[Theorem 2]{cruz2023static}, $(M,g)$ is isometric to $(\mathbb B^3,\delta)$ and $V$ is given by $V(x) =x\cdot v$ for some vector $v \in \mathbb R^3 \setminus \{0\}$.

Conversely, if $(M,g)=(\mathbb B^3,\delta)$, $V(x) = x\cdot v$, then $\Sigma$ is a flat unit disk. We have: $\Lambda=0,~\gamma=0,~b=1,~H_1=2,~K_\Sigma=0,~|\partial_1\Sigma|=2\pi$ and the equality in~\eqref{eq:mineq} is achieved.
\end{proof}

\begin{remark}
The method of proof of Theorem~\ref{thm:ind} is borrowed from the proof of Proposition 17 in~\cite{cruz2024min}, \textit{the Area-Charge Inequality} (see also Theorem D in the same paper and the proof of Theorem 2.2 in~\cite{mendes2018rigidity}, where a similar result was obtained). In fact, Theorem~\ref{thm:ind} is an analog of Proposition 17 in the case, when $Q=0$: \textit{Let $(M,g,V)$ be a compact 3-dimensional static manifold with $R_g= 2\Lambda$. Suppose that $\Sigma=V^{-1}(0)$ is connected and has index one. Then} 
$$
\Lambda|\Sigma|_g\leqslant 12\pi+8\pi\left(\frac{\gamma}{2}-\left[\frac{\gamma}{2}\right]\right).
$$
\textit{Moreover, the equality holds if and only if $\gamma$ is an even integer. In particular, if $\Sigma$ is a two-sphere and $R_g=6$, then $(M,g)$ is isometric to the standard 3-sphere and $\Sigma$ is an equatorial 2-sphere.} The last part follows from Shen's theorem (see~\cite{shen1997note}). 

Finally, the following corollary immediately follows from the proof of Proposition 17 in~\cite{cruz2024min}
\begin{corollary} 
Let $(M,g,V)$ be a 3-dimensional compact static manifold with boundary with $R_g= 2\Lambda$. Suppose that $\Sigma=V^{-1}(0)$, does not intersect $\partial M$ is connected and has index one. Then 
$$
\Lambda|\Sigma|\leqslant 12\pi+8\pi\left(\frac{\gamma}{2}-\left[\frac{\gamma}{2}\right]\right).
$$
Moreover, the equality holds if and only if $\gamma$ is an even integer. 
\end{corollary}
\end{remark}

\begin{remark}
If there is a stable connected component of the zero-level set of the potential on a static manifold with boundary, then all the classical theorems about the stable minimal hypersurfaces in ambient Riemannian manifolds with a given bound on the curvature apply to it. If the zero-level set of the potential is closed and bounds a domain, which does not contain a boundary component, then this domain is a static triple. Then, for example, if the static manifold with boundary is 3-dimensional, has positive scalar curvature, and the zero-level set of the potential is unstable, then by~\cite[Theorem B]{ambrozio2017static} it is a topological sphere and the domain, which is bounded by it, is simply connected. Also, in the above setting (without assuming instability), if we know an estimate of the area of the zero-level set of the potential, which is supposed to be connected, then, for example, the following results also apply~\cite[Theorem D]{ambrozio2017static} and~\cite{shen1997note,boucher1984uniqueness}. Finally, notice that if a connected component $\Sigma$ of the zero-level set of the potential on a 3-dimensional static manifold with boundary is a stable FBMS, then it follows from $S(1,1)\geqslant 0$ that
$$
2\pi\chi(\Sigma) \geqslant \Lambda |\Sigma|+\sum_{i=1}^b\left(k_i+\frac12H_i\right)|\partial_i\Sigma|,
$$
where $k_i$ is the geodesic curvature of $\partial_i\Sigma$ in $\Sigma$. It is not difficult to compute $k_i=H_i/2$, which simplifies the previous inequality. In particular, if $\Lambda$ and $H_i$ for any $i$ are non-negative, then $\Sigma$ is either a topological disk or a cylinder. 
\end{remark}

We pass to the proof of Theorem~\ref{thm:scmc}. We need the following lemma.

\begin{lemma}\label{prop:lambda}
Let $(M^n,g,V)$ be a compact static manifold with boundary and $S$ a connected component of $\partial M$. Assume that $\Ric_g(\nu,\nu)+H_S^2/(n-1)=\lambda_1(S)$, the first Laplace eigenvalue of $S$. Then if $\Sigma=V^{-1}(0)$ intersects $S$, then $\Sigma$ intersects $S$ only once, i.e., the set $\Sigma\cap S$ is connected.
\end{lemma}

\begin{proof}
By~\eqref{eq:comp}, we have
\begin{equation*}
\Delta_SV=-\left(\Ric_g(\nu,\nu)+\dfrac{H_S^2}{n-1}\right)V,
\end{equation*}
which implies that $V$ is the first Laplace eigenfunction. As a corollary of the Courant nodal domain theorem (see Subsection~\ref{sub:spec}), $V$ has only 2 nodal domains. In particular, the set $\Sigma\cap \partial M$ is connected. 
\end{proof}

As a corollary, we obtain Theorem~\ref{thm:scmc}. Here is the proof

\begin{proof}[Proof of Theorem~\ref{thm:scmc}]
Theorem~\ref{cor:closed2} implies that if $\Sigma=V^{-1}(0)$ is not empty, then it intersects $\partial M$. Further, it follows from the contracted Gauss equation that
$$
-\left(\Ric_g(\nu,\nu)+\dfrac{H_g^2}{n-1}\right)=\frac12\left(R_{\partial M}- R_g-\displaystyle\frac{n}{n-1}H^2_g\right)=const,
$$ 
since $R_{\partial M}=const$ by hypothesis of the theorem. Since $\partial M$ is umbilical, $|B|^2_g=H^2_g/(n-1)$. Hence, $\Ric_g(\nu,\nu)+|B|^2_g=const$ on $\partial M$. By~\cite[Proposition 2.13]{barbosa2012stability} $\partial M$ such that $\Ric_g(\nu,\nu)+|B|^2_g=const$ is a \textit{bounding} (i.e., it is a boundary of some ambient Riemannian manifold) stable cmc-hypersurface if, and only if, $\Ric_g(\nu,\nu)+|B|^2_g=\lambda_1(\partial M)$. Then it follows from Lemma~\ref{prop:lambda} that $\Sigma$ intersects $\partial M$ only once. Thus, $\Sigma$ is connected (otherwise, there is another connected component of $\Sigma$ that also intersects the boundary).

Consider the case, when $\Ric_g=0$. Then $R_g=0$. After a harmless homothety, we can assume that $H_g=n-1$. One has
$$
2\lambda_1(\partial M)=-R_{\partial M}+R_g+\frac{n}{n-1}H_g^2=n(n-1)-R_{\partial M}.
$$
The contracted Gauss equation implies $R_{\partial M}=(n-1)(n-2)$. Then $\lambda_1(\partial M)=n-1$. Hence, if we assume that $1$ is the first non-zero Steklov eigenvalue, then $(M,g)$ is isometric to $(\mathbb B^n,\delta)$ by the equality case in the Wang-Xia inequality~\eqref{eq:xw}. Thus, $V$ is a first non-zero Steklov eigenfunction on $(\mathbb B^n,\delta)$. It is well-know that then $V$ is of the form $v\cdot x$, where $v\in \mathbb R^n\setminus\{0\}$.

In order to prove inequality~\eqref{eq:isop} when $n=3$, we remark, that we have already shown that 
$$
\frac12\left(R_{\partial M}- R_g-\displaystyle\frac{3}{2}H^2_g\right)=-\lambda_1(\partial M).
$$
Then~\eqref{eq:isop} follows from the Yang-Yau inequality~\eqref{ineq:YY}. Finally, by \cite[Theorem 1.1]{karpukhin2019yang}, the equality is achieved if, and only if, $\partial M$ is  (up to homotheties) either the standard unit sphere or a Bolza surface, for which $\Lambda_1(\partial M)$ is achieved (see~\cite{nayatani2019metrics}).
\end{proof}

In the previous two theorems the quantity
$$
\Ric_g(\nu,\nu)+H^2_S/(n-1)=\frac12\left(-R_{S}+R_g+n/(n-1)H^2_S\right)
$$
 was a positive constant and $\Sigma$ intersected the boundary. The following proposition essentially states that \textit{if $\Ric_g(\nu,\nu)+H_S^2/(n-1)\leqslant 0$ along a boundary component $S$ of a compact $n$-dimensional static manifold with boundary, then $\Sigma=V^{-1}(0)$ does not intersect $S$. In particular, if $\Ric_g(\nu,\nu)+H_g^2/(n-1)\leqslant 0$, or, which is the same, $R_{\partial M}\geqslant R_g+n/(n-1)H^2_g$ on $\partial M$, then $V^{-1}(0)$ has no free boundary components.}

\begin{proposition}
Let $(M^n,g,V)$ be a compact static manifold with boundary. Suppose that $\Ric(\nu,\nu)+H^2_g/(n-1)$ does not change its sign along a boundary component $S$ and $\Sigma=V^{-1}(0)$ intersects $S$. Then $\Ric_g(\nu,\nu)+H^2_S/(n-1)>0$.
\end{proposition}

\begin{proof}
The proof is also based on the formula
\begin{equation*}
\Delta_SV=-\left(\Ric_g(\nu,\nu)+\dfrac{H_S^2}{n-1}\right)V.
\end{equation*}
Suppose that $\Ric_g(\nu,\nu)+H_S^2/(n-1) <0$. Let $\Omega$ be a connected component of $S\setminus \partial\Sigma$. Then $V$ preserves its sign in $\Omega$. Let $V>0$ on $\Omega$. Then $\Delta_SV >0$, i.e., $V$ is superharmonic on $S$. By the weak maximum principle, $V$ achieves its maximum at $\partial\Omega$. But $V=0$ on $\partial\Omega$. Then $V\equiv 0$, which is impossible, since $V$ can only vanish on $\Sigma \cap S$. The case where $V<0$ is similar. 

Suppose that $\Ric_g(\nu,\nu)+H_S^2/(n-1)=0$. Then $V$ is harmonic on $S$ and since $S$ is compact, $V=const$. This also contradicts the fact that $V$ can only vanish on $\Sigma \cap S$.
\end{proof}

\section{Non-compact static manifolds with boundary}\label{sec:sch}

Previously, we observed that the exterior region of the photon sphere in the Schwarz-schild manifold is a static manifold with boundary. In this section, we explain this phenomenon. We start with the following definition (see, for example, \cite[Section 2]{raulot2021spinorial}).    

\begin{definition}
We say that a triple $(M^n,g,V), \, n\geqslant 3$ is an \emph{asymptotically Schwarz-schildean system of mass $m$}, if 

\begin{itemize}
\item $(M,g)$ is an \emph{asymptotically Schwarzschildean manifold of mass $m$}, i.e., there exists a compact set $U$ in $M$, such that $M\setminus U=\sqcup_{k=1}^KE_k$, where $E_k$ (called \textit{ends}) is diffeomorphic to $\mathbb R^n\setminus \overline{\mathbb B}$, and there exists $m\in \mathbb R$ such that  
$$
(\Phi_{k_*}g)_{ij}-(g_m)_{ij}=O_2(r^{1-n}),~\text{as } r\to\infty,
$$
where $\Phi_k$ is a diffeomorphism between $E_k$ and $\mathbb R^n\setminus\mathbb B$ and $g_m=\left(1+\dfrac{m}{2r^{n-1}}\right)^{\frac{4}{n-2}}\delta$;\\
\item $V$ is an \emph{asymptotically Schwarzschildean potential of the same mass $m$}, i.e., $V-V_m=O_2(r^{1-n})$, where 
$$
V_m(r)=\frac{1-\dfrac{m}{2r^{n-2}}}{1+\dfrac{m}{2r^{n-2}}}.
$$
\end{itemize}
\end{definition}

\begin{remark}
In fact, $V_m$ is the standard Schwarzschildean potential in the \textit{isotropic coordinates}. For this reason, asymptotically Schwarzschildean system is often referred to as \textit{asymptotically isotropic system}.
\end{remark}

We proceed with the following definition (see \cite[Section 2]{raulot2021spinorial} and ~\cite[Definition 15]{jahns2019photon}).

\begin{definition}\label{quasiphoton}
 Let $(M^n, g, V)$ be an asymptotically Schwarzschildean system with compact boundary. Then a connected component $S$  of $\partial M$ is called a \emph{quasi-local photon surface} if it is totally umbilical and it satisfies the following conditions: $V=const>0$ on $S$, the mean curvature of $S$ $H_S$ is a negative constant, and there exists a constant $c_S>1$ such that the equations 
\begin{align}\label{def:eq}
R_{S}=\frac{n-2}{n-1}c_SH_S^2
\quad \text{and} \quad 
     \frac{\partial V}{\partial\nu}=\frac12\frac{n-2}{n-1}(c_S-1)H_SV \quad \text{on}~S
\end{align}
are satisfied.
\end{definition}

\begin{remark}
In \cite[Definition 15]{jahns2019photon} and \cite{raulot2021spinorial} quasi-local photon surfaces are defined such that $H_S$ is a positive constant. The reason is that we consider the outward pointing unit normal $\nu$, while in \cite[Definition 15]{jahns2019photon} and \cite{raulot2021spinorial} $\nu$ points toward the asymptotic end.
\end{remark}

The following lemma holds.

\begin{lemma}\label{thm:photon}
 Let $(M^n,g,V)$ be an asymptotically Schwarzschildean system, which is a static manifold with compact boundary. Let $S$ be a connected component of $\partial M$ such that $V=const> 0$ on $S$ and $H_S<0$. Then $S$ is a quasi-local photon surface.
\end{lemma}

\begin{proof}
Since $V\neq 0$ on $S$, it follows from the first equation of~\eqref{static2} that $S$ is umbilical. Hence, according to definition~\ref{quasiphoton}, we need only to verify that equations~\eqref{def:eq} are satisfied on $S$ for some constant $c_S>1$. In fact, we show that $c_S=\dfrac{n}{n-2}$, i.e., the equations
\begin{equation}\label{eq:photon1}
R_{S}=\frac{n}{n-1}H_{S}^2
\end{equation}
and 
\begin{equation}\label{eq:photon2}
\frac{\partial V}{\partial\nu}=\frac{H_{S}}{n-1}V
\end{equation}
are satisfied on $S$. But equation~\eqref{eq:photon2} holds by the definition of a static manifold with boundary (see the second equation of~\eqref{static2}). So, it remains to verify~\eqref{eq:photon1}.

The contracted Gauss equation and the umbilicity of $S$ imply
\begin{equation}\label{eq:gauss}
R_g-2\Ric_g(\nu,\nu)=R_{S}-\frac{n-2}{n-1}H_{S}^2.
\end{equation}
Further, consider the equation
$$
\Delta_{S}V=-\left(\Ric_g(\nu,\nu)+\dfrac{H_{S}^2}{n-1}\right)V.
$$
Since $V=const$ on $S$, $\Ric_g(\nu,\nu)=-\dfrac{H_{S}^2}{n-1}$. Substituting this into~\eqref{eq:gauss}, we get
$$
R_{S}=\frac{n-2}{n-1}H_{S}^2-2\Ric_g(\nu,\nu)=\frac{n}{n-1}H_{S}^2,
$$ 
which is exactly equation~\eqref{eq:photon1}.
\end{proof}

Now we are ready to prove Theorem~\ref{thm:noncomp}. 

\begin{proof}[Proof of Theorem~\ref{thm:noncomp}]
If $V\neq 0$ on $M\cup\partial M$, then since $V$ is an asymptotically Schwarzschildean potential, $V>0$ on $M$ and $V=const>0$ on $\partial M$. Moreover, since $(M,g)$ is asymptotically Schwarzschildean and $(M,g,V)$ is static, $R_g=const=0$ and $V$ is harmonic. Then the first statement follows from Lemma~\ref{thm:photon} and the Schwarzschild photon sphere uniqueness theorems~\cite[Theorem 4.2]{cederbaum2021photon},\cite[Theorem 3]{jahns2019photon}.

If $n=3$ and $\Sigma=V^{-1}(0)\subset Int~M$ is compact and separates the boundary and the end, then we cut $M$ along $\Sigma$. Consider the closure of the unbounded part of $M\setminus \Sigma$. As in the previous part, we conclude that $R_g=const=0$, $V$ is harmonic, $V>0$ on it. Moreover, the boundary is totally geodesic (since $\Sigma$ is the zero-level set of the potential) and $\nu(V)=const\neq 0$ on it. Then by~\cite[Corollary 4]{jahns2019photon} or \cite[Theorem 4.2]{cederbaum2021photon} (the static horizon case) the closure of the unbounded part of $M\setminus \Sigma$ is isometric to $Sch^3_{-}$. Particularly, the boundary (i.e., $\Sigma$) is connected and isometric to the round sphere $\{2m\}\times \mathbb S^{2}$. Then the boundary of the closure of the bounded part of $M\setminus \Sigma$, that we denote $\Omega$, has two connected components: One of them, which is the zero-level set $\Sigma$ of the potential, is the totally geodesic round sphere $\{2m\}\times \mathbb S^{2}$ and the other one is a mean convex round sphere, which we denote by $S$. Indeed, let $S$ have many connected components. Since $R_g=0$ and the mean curvature of $S$ is positive, then the argument with the associated singular Einstein manifold as in Theorem~\ref{thm:rigidity} $(ii)$ implies that $S$ is connected and it is a round sphere. Consider $\Omega$. The potential $V$ is negative on it. The rest of the proof follows the same arguments as in the end of the proof of Theorem 3 in~\cite{cruz2023static}. For the sake of completeness we repeat it here. Attach the space $Sch^3_-$ to it along $S$ (after a possible rescaling of $m$). Denote the resulting manifold by $M'$. Define a Riemannian metric and a static potential $\widetilde V$ on $M'$ as
$$
\widetilde g=\begin{cases} 
g~&\text{on}~\Omega,\\
g_m~&\text{on}~Sch^3_-
\end{cases} \quad~\text{and} \quad  \widetilde V(x)=\begin{cases} 
V(x),~&\text{if}~x\in\Omega,\\
-V_m(x),~&\text{if}~x\in Sch^3_-.
\end{cases}
$$ 
Here $g_m$ and $V_m$ are the metric and the static potential on $Sch^3_-$. It is not difficult to see that $\widetilde g$ and  $\widetilde{V}$ are smooth away from $S$ and $C^{1,1}$ across it. We can consider $(M', \widetilde g,\widetilde V)$ as a non-compact $C^{1,1}$-static manifold with boundary (see a justification that it is $C^{1,1}$ in~\cite[Section 7.2]{jahns2019photon}). It is scalar flat. Performing the doubling of it by reflection across $\Sigma$, we obtain an asymptotically Schwarzschildean $C^{1,1}$-manifold $(\widehat M,\widehat g, \widehat V)$. It has two ends. Further, we perform a conformal change of the metric $\widehat g$ as $\left(\dfrac{\widehat V+1}{2}\right)^{4}\widehat g$. A standard argument shows that the conformal factor of that metric vanishes nowhere. Namely, notice that it suffices to show that $0\leqslant -\widetilde V<1$ on $\Omega$. Indeed, since $V$ is harmonic on $\Omega$, by the maximum principle it attains its minimum on $\partial\Omega$. But $V=0$ on $\Sigma$ and $V<0$ on $\Omega$. Hence, $V$ attains its minimum on $S$. Thus, $V\geqslant \min_{S}V=-\max_{S} V_m>-1$, which implies that $0\leqslant -\widehat{V}<1$ on $M$ and $0\leqslant \widehat{V}<1$ on its reflected copy. Further, we compactify the non-reflected end of $(\widehat M,\left(\dfrac{\widehat V+1}{2}\right)^{4}\widehat g)$ by adding a point at infinity. We obtain a one-ended, geodesically complete, scalar flat manifold, which is $C^\infty$ away from the gluing surfaces and $C^{1,1}$ along them. Moreover, it has zero ADM-mass (see for example~\cite[Proposition 21]{jahns2019photon}). Then by the rigidity case of Bartnik's positive mass theorem~\cite[Theorem 6.3]{bartnik1986mass}, $(\widehat M,\left(\dfrac{\widehat V+1}{2}\right)^{4}\widehat g)$ is the Euclidean 3-space. Hence, $(M, \widetilde g)$ is $C^\infty$ and conformally flat, a part of which coincides with $Sch^3_-$. Then it follows from~\cite[Theorem 3.1]{kobayashi1982differential}, that this is the doubling of the Schwarzschild 3-manifold. Thus, $\Omega$ is the reflection of $([2m,3m]\times \mathbb S^2,g_m)$, with $m=2/(3\sqrt{3}H_S)$ across the boundary. Therefore, $(M,g)$ is $Sch^3_+$ and $V=V_m$.
\end{proof}

\section{Appendix}\label{appendix}

This section contains supplementary material, which is independent of the previous sections. It concerns conformally flat static manifolds with boundary. Kobayashi in~\cite[Theorem 3.1]{kobayashi1982differential} classified all connected complete conformally flat static manifolds. Based on this result, we obtain the following. 

\begin{theorem}\label{thm:koba}
Any locally conformally flat static manifold with boundary is a domain on one of the following manifolds (up to finite quotients):

(i) round sphere $\mathbb S^n$, Euclidean space $\mathbb E^n$, hyperbolic space $\mathbb H^n$;

(ii) the Riemannian product of the unit circle and the $(n-1)$-round sphere of curvature $k$, $\mathbb S^1\times \mathbb S^{n-1}(k),~k>0$;

(iii)  the Riemannian product $\mathbb R\times N(k), k\neq 0$;

(iv) a warped product $\mathbb R\times_u N(k)$.

Here $N(k)$ denotes an $(n-1)$-dimensional complete Riemannian manifold of constant curvature $k$, which satisfies
$$
k=\dot u^2+\frac{2m}{n-2}u^{2-n}+\frac{R_g}{n(n-1)}u^2,
$$ 
where $m$ is a constant. The warping function is a periodic solution to 
$$
\ddot u+\frac{R_g}{n(n-1)}u=mu^{1-n}.
$$
\end{theorem}

Very recently, Sheng in \cite{sheng2024static} classified domains (possibly non-compact) with one boundary component from the previous theorem. We complete his theorem by the case of domains on the anti-de Sitter-Schwarzschild, de Sitter-Schwarzschild and Nariai spaces (for notation see \textit{Examples 6--10} in Section~\ref{sec:examples}).

\begin{theorem}\label{thm:sheng}
Let $\Omega$ be a locally conformally flat static manifold with one boundary component. 

(i) If $\Omega\subset \mathbb S^n$, then it is a spherical cap. If $\Omega\subset \mathbb E^n$, then it is either a round ball or the exterior of a round ball, or a half-space. If $\Omega\subset \mathbb H^n$, where $\mathbb H^n$ is realized as a hyperboloid model, then it is one of the following domains: a connected component of $\mathbb H^n$ without a horosphere, a geodesic ball, the exterior of a geodesic ball, a connected component of $\mathbb H^n$ without the intersection of $\mathbb H^n$ with a hyperplane in $\mathbb R^{n+1}$.

(ii) If $\Omega\subset \widetilde{Sch^n}$, then it is either the space $Sch^n_+$, or the space $Sch^n_-$, or the space $1/2Sch^n$. 

(iii) If $\Omega\subset \widetilde{AdS-Sch^n}$, then it is either the space $AdS$-$Sch^n_+$, or the space $AdS$-$Sch^n_-$, or the space $1/2AdS$-$Sch^n$. 

(iv) If $\Omega\subset Nar(\mathbb S^{n-1})$, then it is either the space $Nar_{k+}(\mathbb S^{n-1})$, or the space $Nar_{-k}(\mathbb S^{n-1})$, or the space $1/2Nar(\mathbb S^{n-1})$.

(v) If $\Omega\subset \widetilde{dS-Sch^n}$, then it is the space $1/2dS$-$Sch^n$. 

\end{theorem}

\begin{proof}
We only need to proof the third and fourth items. Firstly, consider item $(iii)$. The proof repeats the proof of~\cite[Proposition 4.23]{sheng2024static}. For the sake of completeness, we give it here. We already know from Lemma~\ref{lemma}, that the boundary of a static manifold with boundary is umbilical. The space $\widetilde{AdS-Sch^n}$ is conformal to $\mathbb E^n\setminus \{0\}$ (we consider the isotropic coordinates). Hence, the only possible boundaries are round hyperspheres and hyperplanes. Moreover, the boundaries must be cmc-hypersurfaces. Then the only possibility for hyperspheres is concentric round hyperspheres with the center at the origin. Then it is the photon sphere, as follows from the preliminary computation in Section~\ref{sec:examples}. Similarly, we show that the only possibility for a hyperplane to be cmc in $\widetilde{AdS-Sch^n}$ is to pass through the origin. Indeed, without loss of generality, we may assume that this plane is given by $\{x_n = c\}$. Then we apply the formula of the conformal change of mean curvature in order to find the mean curvature of this hyperplane:
$$
H =-e^{-\varphi}(n-1)\left\langle \nabla^\delta\varphi, \frac{\partial}{\partial x_n}\right\rangle_\delta,
$$
where $\delta$ is the Euclidean metric and $e^{-\varphi}$ is the conformal factor of the $\widetilde{AdS-Sch^n}$-metric. Since the function $\varphi$ is radial, we get
$$
H =-e^{-\varphi}(n-1)\frac{1}{r}\frac{\partial\varphi}{\partial r}x_n.
$$
Further, we see that $e^{-\varphi}\dfrac{1}{r}\dfrac{\partial\varphi}{\partial r}\neq const$ along the plane $\{x_n = c\}$. Hence, the only possibility to get a constant mean curvature is $x_n=0$. This plane is totally geodesic.

Finally, the outward unit normal to the plane $\{x_n=0\}$ is $\nu = e^{-\varphi}\dfrac{\partial}{\partial x_n}$. Then we see
$$
\frac{\partial V}{\partial \nu}=e^{-\varphi}\frac{1}{r}\frac{\partial V}{\partial r}x_n=0,
$$
i.e., equation~\eqref{static2} is satisfied. 

The proof of item $(iv)$ follows exactly the same arguments as the proof of item $(iii)$.

Pass to item $(v)$. Consider the space $dS$-$Sch^n$. It is conformal to a piece of $\mathbb E^n$, bounded by two hyperspheres, centered at the origin. Then, as in the previous case, the only possible boundaries for $\Omega$ are round hyperspheres centered at the origin or hyperplanes passing through the origin. In the case of a hyperplane, we obtain $1/2dS$-$Sch^n$ after the doubling. In the case of a hypersphere we obtain the photon sphere. However, after the doubling we obtain two hyperspheres. Then $\Omega$ is bounded either by two hyperspheres or it is $1/2dS$-$Sch^n$.

\end{proof}

\begin{remark}
It follows from the proof that the remaining cases for $\Omega$ in item $(v)$ are the spaces $dS$-$Sch^n_+$ and $dS$-$Sch^n_-$.
\end{remark}

\begin{remark}
One can obtain domains with many boundary components by taking intersections of domains listed in Theorem~\ref{thm:sheng}.
\end{remark}

\bibliography{mybib}
\bibliographystyle{alpha}

\end{document}